\documentclass[12pt,article]{amsart}
\usepackage[a4paper,textwidth=17.5cm,textheight=21cm,includehead]{geometry}

\usepackage[latin1]{inputenc}

\usepackage{latexsym}
\usepackage{amsthm}
\usepackage{epic}
\usepackage{amsfonts}

\usepackage{rotating}

\usepackage[T1]{fontenc}

\usepackage{amsmath} 
\usepackage{amssymb} 
\usepackage[all]{xy}  
\usepackage{graphicx}  
\usepackage{xcolor}
\usepackage{mathrsfs}
\usepackage{euscript}
\usepackage{listings}

\newtheorem{thm}{Theorem}[section]
\newtheorem{prop}[thm]{Proposition}

\newtheorem{cor}[thm]{Corollary}
\newtheorem{lem}[thm]{Lemma}
\newtheorem{defn}[thm]{Definition}

\newtheorem{remark}[thm]{Remark}
\newtheorem{example}[thm]{Example}

\renewcommand{\O}{\mathrm{O}}
\renewcommand{\phi}{\varphi}

\newcommand{\Wedge}{\bigwedge}
\newcommand{\R}{\mathbb{R}}
\newcommand{\RR}{\mathbb{R}}

\newcommand{\CC}{\mathbb{C}}

\newcommand{\Pe}{\mathbb{P}}

\newcommand{\ZZ}{\mathbb{Z}}

\newcommand{\cE}{\mathcal{E}}
\newcommand{\cF}{\mathcal{F}}

\newcommand{\cO}{\mathcal{O}}
\newcommand{\cQ}{\mathcal{Q}}
\newcommand{\cI}{\mathcal{I}}

\newcommand{\Sec}{\operatorname{Sec}}

\newcommand{\diag}{\mathrm{diag}}
\newcommand{\comment}[1]{}
\newcommand{\codim}{\operatorname{codim}}

\newcommand{\red}[1]{#1}

\makeatletter
      \def\@setcopyright{}
      \def\serieslogo@{}
      \makeatother

\begin{document}

\author{Jan Draisma}
\address[Jan Draisma]{Universit\"at Bern, Mathematisches Institut, Sidlerstrasse 5,
3012 Bern, Switzerland; and Department of Mathematics and Computer
Science, Eindhoven University of Technology, The Netherlands}
\email{jan.draisma@math.unibe.ch}
\thanks{The first author was partially supported by the NWO Vici grant
entitled {\em Stabilisation in Algebra and Geometry}.}

\author{Giorgio Ottaviani}
\address[Giorgio Ottaviani]{Universit\`a di Firenze, Dipartimento di Matematica e
Informatica U. Dini, viale Morgagni 67/A, 50134 Firenze, Italy} 
\email{ottavian@math.unifi.it}
\thanks{The second author is member of INDAM-GNSAGA}

\author{Alicia Tocino}
\address[Alicia Tocino]{
Universidad de M\'alaga, Departamento de \'Algebra, Geometr\'ia y
Topolog\'ia, Bulevar Louis Pasteur, 31, 29010 M\'alaga, Spain}
\email{aliciatocinosanchez@ucm.es}

\begin{abstract}
Given a tensor $f$ in a Euclidean tensor space, we are interested in the critical points of the distance function from $f$ to the set of tensors of 
rank at most $k$, which we call the critical rank-at-most-$k$ tensors for $f$. When $f$ is a matrix, the critical rank-one matrices for $f$ correspond to the singular pairs of $f$.
The critical rank-one tensors for $f$ lie in a linear subspace $H_f$, the critical space of $f$. Our main result is that,
for any $k$, the critical rank-at-most-$k$ tensors for a sufficiently general $f$ also lie in the critical
 space $H_f$. This is the part of Eckart-Young Theorem that generalizes from matrices to tensors.
Moreover, we show that when the tensor format satisfies the triangle inequalities, the critical space $H_f$ is
spanned by the complex critical rank-one tensors. Since $f$ itself belongs to $H_f$, we deduce that also $f$ itself is a linear combination of its critical rank-one tensors.
\end {abstract}



\title[\red{Best rank-$k$ approximations for tensors}]{\red{Best rank-$k$ approximations for tensors:\\ generalizing
Eckart-Young}}

\maketitle

\section{Introduction}

The celebrated Eckart-Young Theorem says that, for a real $m \times
n$-matrix $A$ with $m \leq n$ and for an integer $k \leq m$, a matrix
$B$ of rank at most $k$ nearest to $A$ is obtained from $A$ as follows:
Compute the singular value decomposition $A=U \Sigma V^T$, where $U,V$
are orthogonal matrices and where $\Sigma=\diag(\sigma_1,\ldots,\sigma_m)$
is the ``diagonal'' $m \times n$-matrix with the singular values $\sigma_1
\geq \cdots \geq \sigma_m \geq 0$ on its main diagonal, and set $B:=U
\diag(\sigma_1,\ldots,\sigma_k,0,\ldots,0) V^T$. Such a best rank-$k$
approximation is unique if and only if $\sigma_k > \sigma_{k+1}$,
and for us ``nearest'' refers to the Frobenius norm (but in fact,
the result holds for arbitrary $\O_m \times \O_n$-invariant norms
\cite{Mirsky}).

For higher-order tensors, an analogous approach for finding best rank-$k$
approximations fails in general \cite{VNVM}. It succeeds, with respect
to the Frobenius norm, for {\em orthogonally decomposable tensors}
\cite{VNVM,BDHR}, but this is a very low-dimensional real-algebraic
variety in the space of all tensors.  In this paper, we will establish
versions of the Eckart-Young Theorem and the Spectral Theorem that {\em
do} hold for general tensors.

To motivate this theorem, consider matrices once again, and
assume that the $\sigma_i$ are distinct and positive. A statement
generalizing the Eckart-Young Theorem says that we obtain all critical
points of the distance function $d_A(B):=||A-B||^2$ on the manifold of
rank-$k$ matrices by setting any $m-k$ of the singular values equal to
zero \cite{DHOST}, so as to obtain a matrix
\[ B_{i_1,\ldots,i_k}:=U
\diag(0,\ldots,0,\sigma_{i_1},0,\ldots,0,\sigma_{i_2},0,
\ldots,0, \sigma_{i_k},0,\ldots) V^T \]
for any ordered $k$-tuple $i_1<\ldots<i_k$ in $\{1,\ldots,m\}$. We will
call these critical points {\em critical rank-$k$ matrices for $A$}.
In particular, the critical rank-one matrices are $B_1,\ldots,B_m$, and
we draw attention to the fact that for each $k \geq 1$ and each $k$-tuple
$i_1<\ldots<i_k$ the critical rank-$k$ matrix $B_{i_1,\ldots,i_k}$ lies
in the linear span of $B_1,\ldots,B_m$. Moreover, this linear span has
a direct description in terms of $A$: it consists of all matrices $B$
such that both $A^TB$ and $AB^T$ are symmetric matrices.

Taking cue from this observation, we will associate a {\em critical space}
$H_f$ to a tensor $f$,
show that $H_f$ contains the critical \red{rank-at-most-$k$}
tensors for $f$ for each value of $k$ (see below for a definition),
and that $H_f$ is spanned by the critical rank-one tensors for $f$. We
will establish these results for sufficiently general partially symmetric tensors, and we work over the base field $\CC$ rather than $\RR$.

\begin{thm} \label{thm:Main}
Let $f$ be a sufficiently general tensor in $S^{d_1} \CC^{n_1+1} \otimes
\cdots \otimes S^{d_p} \CC^{n_p+1}$. Then for each natural number $k$,
the critical \red{rank-at-most-$k$} tensors for $f$ lie in the critical
space $H_f$.  Moreover, if for each $\ell$ with $d_\ell=1$ the {\em
triangle inequality} $n_\ell\leq\sum_{i \neq \ell} n_i $ holds, then
$\codim H_f=\sum_\ell \binom{n_\ell + 1}{2}$ and $H_f$ is spanned by
the critical rank-one tensors for $f$.
\red{In particular, $f$ itself lies in the linear
span of the critical rank-one tensors for $f$.}
\end{thm}

We record the following two corollaries over the real numbers.

\begin{cor} \label{cor:tensor}
If $n_1,\ldots,n_p$ satisfy the triangle inequality $n_\ell \leq \sum_{i
\neq \ell} n_i$ for each $\ell=1,\ldots,p$, then for a sufficiently
general tensor $f \in \bigotimes_{i=1}^p \RR^{n_i+1}$ and any natural
number $k$, any real tensor of real rank at most $k$ closest to $f$
in the Frobenius norm lies in the linear span of the complex critical
rank-one tensors for $f$. In particular, $f$ itself lies in the linear
span of the complex critical rank-one tensors for $f$.
\end{cor}

\begin{cor}\label{cor:sym}
For a sufficiently general symmetric tensor $f \in S^d \RR^{n+1}$ and any
natural number $k$, any real symmetric tensor of real symmetric rank at
most $k$ closest to $f$ in the Frobenius norm lies in the linear span of
the complex critical symmetric rank-one tensors for $f$.  In particular,
$f$ itself lies in the linear span of the complex critical rank-one
tensors for $f$.
\end{cor}

In the case of symmetric tensors, these critical rank-one tensors
correspond to the so-called {\em eigenvectors} of $f$ \cite{Mac},
while in the case of ordinary tensors, they correspond to {\em singular
vector tuples} \cite{Lim}. In the case $n=1$ of binary forms, Corollary
\ref{cor:sym} was proved in \cite{OT}. The two corollaries above can be
regarded as generalizations of the Eckart-Young Theorem and the Spectral
Theorem from matrices to tensors.

Several remarks are in order here. First, we complexify $d_f$ to
the quadratic polynomial $d_f(u):=(u-f|u-f)$, where $(.|.)$ is the
standard complex bilinear form on the space of tensors (and not a
Hermitian form). The point of doing this is that, unlike for matrices,
the critical points of this function on low-rank tensors are in general
not real anymore, even if $f$ is real. Accordingly, the critical space
$H_f$, while defined by linear equations over $\RR$ if $f$ is real, is
taken to be the space of complex solutions to those equations. Second,
we denote the dimensions by $n+1$ rather than $n$ since we will be using
methods from projective algebraic geometry where the formulas look
more appealing in terms of the projective dimension $n$ than in the
affine dimension $n+1$. An example of this phenomenon is the triangle
inequalities in the theorem, which hold if and only if the variety dual
to the Segre-Veronese embedding of the product $\Pe^{n_1} \times
\cdots \times \Pe^{n_p}$ via degrees $d_1,\ldots,d_p$ 
is a hypersurface \cite[Corollary 5.11]{GKZ}.

The remainder of this paper is organized as follows. In
Section~\ref{sec:CritSpace} we define the critical space
$H_f$ for a partially symmetric tensor $f$ and prove that the
critical \red{rank-at-most-$k$} tensors for $f$ lie in $H_f$. Then, in
Section~\ref{sec:Span} we use vector bundle techniques to compute the
dimension of the space spanned by the critical rank-one tensors for $f$
and to show that this space equals $H_f$. In Section~\ref{sec:Final}
we combine these ingredients to establish the results above.

\section{The critical space of a tensor}
\label{sec:CritSpace}

\subsection*{Partially symmetric tensors and their ranks}

Let $p \in \ZZ_{\geq 1}$, let $V_1,\ldots,V_p$ be complex vector spaces,
and let $d_1,\ldots,d_p \in \ZZ_{\geq 1}$. Let $S^dV$ be the $d$-th symmetric power of $V$. We will study
tensors in the space 
\[ T:=S^{d_1} V_1 \otimes \cdots \otimes S^{d_p} V_p. \]
So for $p=1$, $T$ is a symmetric power of $V_1$, which is canonically
isomorphic with the space of symmetric, $d_p$-way $n_1 \times \cdots
\times n_1$-tensors. On the other hand, when all $d_i$ are equal to $1$,
then $T$ is a space of $p$-way ordinary tensors. We will
write $[p]:=\{1,\ldots,p\}$. 

Inside $T$, let $X$ be the set of all tensors of the form 
\[ v_1^{d_1} \otimes \cdots \otimes v_p^{d_p}\ \quad 
(v_\ell \in V_\ell, \ell \in [p]).\]
Then $X$ is a closed subvariety of $T$ known as the affine cone over the
Segre-Veronese embedding of $\Pe^{n_1} \times \cdots \times \Pe^{n_p}$
of degrees $(d_1,\ldots,d_p)$. Let $kX$ denote the set of sums of $k$
elements of $X$.  An arbitrary element $t$ of $T$ lies in $kX$ for some $k$,
and the minimal such $k$ is called the rank of $t$ \cite[Definition
5.2.1.1]{Landsberg12}. For $p=1$
this is the symmetric or Waring rank, and if all $d_q$ are $1$, this notion is
ordinary tensor rank. We write $\Sec_k(X)$ for the Zariski
(or Euclidean) closure of $kX$ in $T$.

For real tensors a few modifications are needed. The real rank of a real
tensor $t$ is the minimum $k$ such that $t=\sum_{i=1}^k\lambda_ix_i$
with $\lambda_i\in{\mathbb R}$ and $x_i\in X_{\R}$ (it is enough to allow $\lambda_i=\pm 1$). For example {$(e_1+ie_2)^3+(e_1-ie_2)^3$} has complex rank $2$ and real rank $3$. Real rank is subtle for low-rank approximation of tensors.
A classical example of de Silva and Lim \cite{dSL} shows that for almost all $2\times 2\times 2$-tensors of real rank $3$ (like the above one) does not exist a closest tensor
of real rank $2$, while such phenomena may happen only for measure zero subsets of
the set of complex tensors of given rank.

\subsection*{Symmetric bilinear forms and pairings}

If $V,W$ are complex vector spaces with symmetric bilinear forms $(.|.)$,
and if $d \in \ZZ_{\geq 0}$, then $S^d V$ and $V \otimes W$ carry unique
symmetric bilinear forms, also denoted $(.|.)$, that satisfy

\begin{align*} 
(v^d|v'^d)&:=(v|v')^d\text{ and} \\
(v \otimes w | v' \otimes w') &:= (v|v') (w|w').
\end{align*}

{The first of these equalities implies

\begin{align*}
(v_1\ldots v_d|v'^d)&=\prod_{i=1}^d(v_i|v')
\end{align*}
}
and more in general
\begin{align*}
(v_1 \cdots v_d | v'_1 \cdots v'_d)&= \frac{1}{d!} \sum_{\pi \in S_d}
\prod_{i=1}^d (v_i|v'_{\pi(i)}) 
\end{align*}

We now fix nondegenerate symmetric bilinear forms on each $V_\ell,\
\ell \in [p]$. Then, iterating
these constructions, we obtain a canonical bilinear form on $T$.

Using the bilinear forms on $V$ and $W$, we can also build more
general bilinear maps whose outputs are vectors or tensors rather than
scalars. We will call these bilinear maps {\em pairings} and denote them
by $[.|.]$. Of particular relevance to us is the skew-symmetric pairing $S^d V \times
S^d V \to \Wedge^2 V$
determined by

\[ 
[v^d | w^d]:=
 (v|w)^{d-1}  
v \wedge w
\]

which implies

{
\[
[v_1\ldots v_d | w^d]=\frac{1}{d}\sum_{i'\in[d]}
\left(\prod_{i\neq i'}(v_i|w)\right)v_{i'}\wedge w
\]
}

and more in general

\[ 
[v_1 \cdots v_d | w_1 \cdots w_d]=
\frac{1}{d\cdot d!}\sum_{i',j' \in [d]} \sum_{\pi:[d]\setminus i' \to
[d]\setminus j'}
\left( \prod_{i \neq i'} (v_i|w_{\pi(i)}) \right) 
v_{i'} \wedge w_{j'}
\]
where $\pi$ runs over all bijections $[d]\setminus i' \to [d]\setminus
j'$. 

\begin{remark}
In the case of binary forms ($\dim V=2$ and arbitrary $d$), the pairing $[f|g]$
coincides (up to scalar multiples) with $\left( f| D(g)\right)$, where
	$D(g)=g_xy-g_yx$\red{;} see \cite{OT}. Note the skew-symmetry property $\left(
f| D(g)\right) = -\left( g| D(f)\right)$. On the other hand, in the case
of symmetric matrices ($d=2$ and arbitrary $V$), the pairing $[f|g]$
coincides (up to scalar multiples) with the bracket $fg-gf$.
\end{remark}

Building on this construction, for each $\ell \in [p]$ we define
a pairing $[.|.]_\ell:T \times T \to \Wedge^2 V_\ell$ by
\begin{equation} \label{eq:Pairing} 
[f_1 \otimes \cdots \otimes f_p|g_1 \otimes \cdots \otimes
g_p]_\ell:=\left(\prod_{i \neq \ell} (f_i|g_i)\right) [f_\ell|g_\ell],\
f_i,g_i \in S^{d_i} V_i.
\end{equation}
which we will use to define the critical space. 

\begin{remark}\label{rem:bracket2}In the case of matrices $T=V_1\otimes V_2$, the pairing $[f,g]_1$
coincides (up to scalar multiples) with $fg^t-gf^t$,
while  $[f,g]_2$
is (up to scalar multiples) $f^tg-g^tf$.
\end{remark}

{
\subsection*{Basis, orthogonal basis and monomials}

If $B$ is a basis of $V$, then the degree-$d$ monomials in
the elements of $B$ form a basis of $S^d V$.
Such a basis is orthogonal if $B$ is orthogonal.
Hence if we fix bases (respectively, orthogonal bases) of $V_1,\ldots,V_p$, then by taking
tensor products we obtain a basis (respectively, orthogonal basis) of $T$, whose elements we
will call monomials of degree $D:=\sum_{\ell=1}^p d_\ell$.} We will use
the word {\em gcd} of two such monomials $x,y$ for the highest-degree
monomial $z$ such that both $x$ and $y$ can be obtained from $z$ by
multiplying $z$ with suitable monomials.

\begin{example}
If $p=3$ and $V_1=V_2=V_3=\CC^3$ with the standard bilinear
form, and $d_1=d_2=3$ and $d_3=2$, then the gcd of 
$(e_1^2 e_2) \otimes (e_1e_2e_3) \otimes (e_1^2)$ and
$(e_1 e_2 e_3) \otimes (e_3^3) \otimes (e_2e_3)$ equals 
$(e_1 e_2) \otimes (e_3) \otimes (1)$. 
\end{example}

\begin{lem} \label{lem:Pairing}
For two monomials $f=f_1 \otimes \cdots \otimes f_p,g=g_1 \otimes
\cdots \otimes g_p$ in $T$ relative to the same orthogonal bases of
$V_1,\ldots,V_p$ and for $\ell \in [p]$
we have $[f|g]_\ell=0$ unless $f_i=g_i$ for all $i \neq \ell$ and
$h:=\gcd(f_\ell,g_\ell)$ has degree $d_\ell-1$; in this case $\gcd(f,g)$
has degree $D-1$ and $[u|v]_\ell \in \CC^* (f_\ell/h) \wedge (g_\ell/h)$. 
\end{lem}

This is immediate from the definition of the pairing in
\eqref{eq:Pairing}.

\subsection*{Critical rank-one tensors}

Let $f \in T$. Then the critical points of the distance function $d_f:
x \mapsto (f-x|f-x)$ on $X$ are by definition those $x \in X \setminus
\{0\}$ for which $f-x$ is perpendicular to the tangent space $T_x X$ to
$X$ at $x$; we write this as $f-x \perp T_x X$. We call these tensors the
{\em critical rank-one tensors} for $f$. For sufficiently general $f$,
each of these critical rank-one tensors is {\em non-isotropic}, i.e.,
satisfies $(x|x) \neq 0$ (see \cite[Lemma 4.2]{DLOT}, in next Proposition
\ref{prop:RankOne} we will prove a slightly more general fact).

\comment{
Slightly
more generally, we call a non-isotropic $x$ a {\em singular rank-one
tensor} for $f$ if $f-cx \perp T_x X$ for some $c \in \CC$.  In that
case, since $x$ is non-isotropic and since $x \in T_xX$, the $c$ with
this property is the unique $c$ with $(f-cx|x)=0$; and if $c \neq 0$,
then $cx \in X$ is a critical point of $d_f$.
}

We will establish a bilinear characterization of these critical rank-one
tensors for $f$. First, we describe the tangent space of $X$ {at a point $x$} in more
detail.
For this, write
\begin{equation} \label{eq:RankOne}
x=v_1^{d_1} \otimes \cdots \otimes v_p^{d_p}. 
\end{equation}
  {Hence we may
extend each $v_\ell$ to a basis of $V_\ell$. We then obtain
an $x$-{\em adapted} basis of $T$ consisting of monomials.
If moreover $x$ is non-isotropic, we have $(v_\ell|v_\ell) \neq 0$ and we may extend each $v_\ell$ to an orthogonal basis.
 We then obtain
an $x$-{\em adapted} orthogonal basis of $T$.
}

{
\begin{lem} \label{lem:Tangent}
Let $x\in X$ as in \eqref{eq:RankOne}.
\begin{enumerate}
\item   Then, relative
to any $x$-adapted basis, $T_x X$ is spanned by all degree-$D$
monomials whose gcd with $x$ has degree at least $D-1$.

\item  Assume moreover that $x$ is non-isotropic. Then, relative
to any $x$-adapted orthogonal basis, $(T_x X)^\perp$ is spanned by all degree-$D$
monomials whose gcd with $x$ has degree at most $D-2$.
\end{enumerate}
\end{lem}

\begin{proof}
Part (1) follows by applying the Leibniz rule to the
parameterisation \eqref{eq:RankOne}
of $X$; part (2) is a straightforward consequence.
\end{proof}
}

\begin{prop} \label{prop:RankOne}
Let $f \in T$ and let $x \in X$ be
non-isotropic. Then the following two statements are equivalent:
\begin{enumerate}
\item some (nonzero) scalar multiple of $x$ is a critical rank-one tensor for
$f$ and
\item a unique (nonzero) scalar multiple of $x$ is a critical rank-one tensor
for $f$; 
\end{enumerate}
and they imply the following statement:
\begin{enumerate}
\addtocounter{enumi}{2}
\item for each $\ell \in [p]$, $[f|x]_\ell \in \Wedge^2 V_\ell$ is
zero.
\end{enumerate}
Moreover, if $f$ is sufficiently general, then {\em every} nonzero $x
\in X$ satisfying (3) is non-isotropic and satisfies (1) and (2).
\end{prop}

The pairing in item (3) is the pairing from \eqref{eq:Pairing}. 

\begin{proof}
For the equivalence of the first two statements, we note that if $cx,c'x$
with $c,c' \neq 0$ are critical rank-one tensors for $f$, then $T_{cx}
X = T_{c'x} X=T_x X$ and $f-cx \perp T_x X$ and $f-c'x \perp T_x X$.
Since $x \in T_x X$, we find that $(c-c')x \perp x$, and using that $x$
is non-isotropic we find that $c=c'$.

To prove that (1) implies (3), write $x$ as in \eqref{eq:RankOne} and
extend each $v_\ell$ to an orthogonal basis of $V_\ell$, so as to obtain
an $x$-adapted orthogonal basis of $T$. Now assume that $f-cx \perp T_x
X$. Then by Lemma~\ref{lem:Tangent}, $f-cx$ is a linear combination of
degree-$D$ monomials whose gcds with $x$ have degrees at most $D-2$.
Hence by Lemma~\ref{lem:Pairing}, $[x|f-cx]_\ell=0$. By the skew-symmetry,
$[x|x]_\ell=0$, so $[x|f]_\ell=0$.

For the last statement, consider an $x=v_1^{d_1} \otimes \cdots
\otimes v_p^{d_p} \in X$ where, say, $v_1,\ldots,v_a$ with $a>0$ are isotropic
but the remaining factors are not. Extend each $v_\ell$, $\ell > a$
to an orthogonal basis of $V_\ell$, and for $v_\ell$ with $\ell \leq a$
find an isotropic $w_\ell \in V_i$ with $(v_\ell|w_\ell) =1$ and extend
$v_\ell,w_\ell$ with an orthogonal basis of the orthogonal complement
of $\langle v_\ell, w_\ell \rangle^\perp$ to a basis of $V_\ell$.  In the
corresponding (non-orthogonal) monomial basis of $T$, the monomials $y$
with $[y|x]_\ell \neq 0$ for $\ell \leq a$ are those of the form
\[ w_1^{d_1} \otimes \cdots \otimes w_\ell^{d_\ell-1} u_\ell \otimes
\cdots \otimes w_a^{d_a} \otimes v_{a+1}^{d_{a+1}} \otimes \cdots
\otimes v_p^{d_p} \]
where $u_\ell$ is a basis vector of $V_\ell$ that is distinct from
$v_\ell$ but possibly equal to $w_\ell$. These monomials all satisfy
$[y|x]_i=0$ for $i \neq \ell$. Similarly, the monomials $y$ with
$[y|x]_\ell \neq 0$ for $\ell > a$ are those of the form 
\[ w_1^{d_1} \otimes \cdots \otimes w_a^{d_a} \otimes
v_{a+1}^{d_{a+1}} \otimes \cdots \otimes v_l^{d_l-1} u_l \otimes
\cdots \otimes v_p^{d_p} \]
with $u_\ell$ a basis vector of $V_\ell$ distinct from $v_\ell$; 
they, too, satisfy $[y|x]_i=0$ for $i \neq \ell$. 
The remaining monomials span the space of $f$s with $[x|f]_\ell=0$
for all $\ell$; this space therefore has dimension
\[ \dim T - (n_1+\ldots+n_p), \]
and it does not change when we scale $x$. Since the isotropic projective
points $\langle x \rangle \in \Pe T$ form a subvariety of positive
codimension in the $(n_1+\ldots+n_p)$-dimensional projective variety
$\Pe X$, the locus of all $f$ for which there is a nonzero isotropic $x
\in X$ with $[f|x]_\ell=0$ for all $\ell$ has dimension less than
$\dim T$.

Now assume that $f$ is sufficiently general and let $x \in X \setminus
\{0\}$ satisfy $[x|f]_\ell=0$ for all $\ell$. By the above, $x$ is
non-isotropic. Suppose that $f$, expanded on the $x$-adapted orthogonal
basis, contains a monomial $y$ whose gcd with $x$ has degree exactly
$D-1$. If $y$ agrees with $x$ except in the factor $S^{d_\ell} V_\ell$
where it equals $v_\ell^{d_\ell-1} u_\ell$, then in $[x|f]_\ell$,
expanded on the standard basis of $\Wedge^2 V_\ell$ relative to the
chosen basis of $V_\ell$, the term $v_\ell \wedge u_\ell$ has a nonzero
coefficient. Hence $[x|f]_\ell$ is nonzero, a contradiction.

Therefore, $f$ contains only monomials whose gcds with $x$ have degrees
at most $D-2$, and possibly the monomial $x$ itself. Then $f-c x \perp
T_x X$ for a unique constant $c$. By generality of $f$, it does not lie
in $(T_x X)^\perp$ for any $x \in X \setminus \{0\}$ (the union of these
orthogonal complements is the cone over the variety dual to the projective
variety defined by $X$, and of positive codimension). Hence $c \neq 0$,
and $cx$ is a critical rank-one tensor for $f$.
\end{proof}

{
\begin{remark}
The implication (1) $\Longrightarrow$ (3) in Proposition \ref{prop:RankOne} holds without the assumption of non-isotropy of $x$.
This follows from the fact that the ED correspondence $$\{(x,f)\in X\times V|\ x\textrm{\ is critical for\ }f\}$$
is a irreducible variety (see \cite[\S 4 and Lemma 2.1]{DHOST}) and the nonempty open part in it where $x$ is non-isotropic
lies in the variety defined by $[f|x]_\ell=0\ \forall\ell \in [p]$ by Proposition \ref{prop:RankOne}.
\end{remark}
}

\subsection*{The critical space}

In view of Proposition~\ref{prop:RankOne}, we introduce the following
notion.

\begin{defn}
For a tensor $f \in T$, {\em the critical space $H_f \subseteq T$ of $f$} 
is defined as 
\[ H_f:=\{g \in T \mid [f|g]_\ell=0 \text{ for all } \ell \in [p]\}.
\]
\end{defn}

\begin{remark}\label{rem:fincritical} By the skew-symmetry, it follows immediately that
$f\in H_f$.
\end{remark}

\begin{remark}
In the case of binary forms ($\dim V=2$), $H_f$ is the hyperplane orthogonal to $D(f)$ \cite{OT}. In the case of ordinary tensors,
$H_f$ was first defined in \cite{OP} where it was called singular space, but
in view of the results in this paper we feel that {\em critical space}
is a better name.
\end{remark} 

Proposition~\ref{prop:RankOne} establishes that the non-isotropical
critical rank-one tensors all lie inside $H_f$; hence for a sufficiently
general $f$, all critical rank-one tensors lie in $H_f$.  In the next
subsection we will establish an analogous statement for higher ranks.

Note that the number of linear conditions for $g$ to lie in $H_f$
is at most $\sum_{\ell=1}^p \dim \Wedge^2 V_\ell = \sum_{\ell=1}^p
\binom{n_\ell+1}{2}$---the linear conditions in the definition may not
all be linearly independent. In Proposition~\ref{prop:dim} we will see
that, assuming the triangle inequalities from Theorem~\ref{thm:Main} and
assuming that $f$ is sufficiently general, equality holds.

\subsection*{Higher rank}

We will now establish a generalization of Proposition~\ref{prop:RankOne}
to higher-rank tensors. 

\red{
\begin{defn}
Let $f \in T$ and let $k$ be any nonnegative integer.
A {\em critical rank-at-most-$k$ tensor} for $f$
is a tensor $g \in kX$ such that $f-g \perp T_g \Sec_k(X)$.
\end{defn}
}

\red{Note that by \cite[Lemma 4.2]{DLOT}, all the critical
rank-at-most-$k$ tensors for a sufficiently general $f \in T$ are smooth
points of $\Sec_k(X)$ and can be written as a sum of $k$ non-isotropic
rank-one tensors. Moreover, if we assume that $k$ is at most the generic
rank of tensors in $T$, then these critical tensors to a sufficiently
general $f$ have rank equal to $k$. If $k$ is at least the generic
rank of tensors in $T$, then the only critical
rank-at-most-$k$ tensor for a sufficiently general $f$ is $f$ itself.}

\begin{prop} \label{prop:HigherRank}
Let $f \in T$ be sufficiently general \red{and let $k$ be a nonnegative
integer}. Then all the critical \red{rank-at-most-$k$} tensors for $f$ lie in
the critical space $H_f$.
\end{prop}

\begin{proof}
Let $g$ be a critical \red{rank-at-most-$k$} tensor. By generality of $f$, $g$
can be written as $x_1+\cdots+x_k$ with each $x_i \in X$ non-isotropic.
Then $T_g \Sec_k X \supseteq \sum_{i=1}^k T_{x_i} X$, so that for each $i \in
[k]$ we have $f-g \perp T_{x_i} X$. By
Lemma~\ref{lem:Tangent} this means that, in the $x_i$-adapted
orthogonal basis, $f-g$ is a linear combination of monomials whose gcds
with $x_i$ have degrees at most $D-2$. Hence, by Lemma~\ref{lem:Pairing},
$[f-g|x_i]_\ell=0$ for all $\ell=1,\ldots,p$. We conclude that, for
each $\ell$,
\[ [f-g|g]_\ell=\sum_{i=1}^k [f-g|x_i]_\ell=0, \]
and therefore 
\[ [f|g]_\ell=[f-g|g]_\ell + [g|g]_\ell=0 + 0, \]
where in the last step we used that $[.|.]_\ell$ is
skew-symmetric. Hence $g \in H_f$.
\end{proof}

In the next section we compute the dimension of the space spanned by
the critical rank-one tensors for a general tensor, and show that this
space equals $H_f$.

\section{The scheme of critical rank-one tensors}
\label{sec:Span}

\subsection*{Critical rank-one tensors as the zero locus of a vector
bundle section}

Let $f \in T=\bigotimes_{\ell=1}^p S^{d_\ell} V_\ell$ be a tensor. We
assume that $p \geq 2$, $d_\ell \geq 1$, and $\dim V_\ell=n_\ell+1 \geq 1$
for all $\ell$. We adapt the notation of \cite[Section 5.1]{OP} to our
current setting.

Consider the Segre-Veronese variety $\Pe X=\Pe V_1\times\ldots\times\Pe
V_p$ embedded with $\cO(d_1,\ldots,d_p)$ in $\Pe T$; so $\Pe X$
is the projective variety associated to the affine cone $X \subseteq T$.
Let $\pi_\ell:\Pe X \rightarrow \Pe V_\ell$ be the projection on the
$\ell$-th factor and set $N:=\dim \Pe X = n_1+\ldots+n_p$. For each $\ell
\in [p]$ let $\cQ_\ell$ be the quotient bundle on $\Pe V_\ell$, whose
fibre over a point $\langle v \rangle$ is $V_\ell/\langle v \rangle$. From
these quotient bundles, we construct the following vector bundles on
$\Pe X$:
\[ \cE:=\bigoplus_{\ell=1}^{p} \cE_l \quad \text{where}
\quad \cE_l:= (\pi_\ell^*\mathcal{Q_\ell})\otimes
\cO(d_1,\ldots,d_{\ell-1},d_\ell-1,d_{\ell+1},\ldots,d_p). \]
Note that $\cE$ has rank $N$.  The fibre of $\cE_\ell$ over a point
$v:=(\langle v_1 \rangle,\ldots,\langle v_p \rangle) \in \Pe X$
consists of polynomial maps $\prod_{i=1}^p \langle v_i \rangle
\to V_\ell/\langle v_\ell \rangle$ that are multi-homogeneous of
multi-degree $(d_1,\ldots,d_\ell-1,\ldots,d_p)$.  The tensor $f$
yields a global section of $\cE_\ell$ which over the point $v$ is
the map sending $(c_1 v_1,\ldots,c_p v_p)$ to the natural pairing of
$f$ with $(c_1 v_1)^{d_1} \cdots (c_\ell v_\ell)^{d_\ell-1} \cdots
(c_p v_p)^{d_p}$---a vector in $V_\ell$---taken modulo $\langle v_\ell
\rangle$. Combining these $p$ sections, $f$ yields a global section $s_f$
of $\cE$. By Proposition~\ref{prop:RankOne}, for $f$ sufficiently general,
the tensor $x:=v_1^{d_1} \otimes \cdots
\otimes v_p^{d_p}$ is a nonzero scalar multiple of a critical rank-one
tensor for $f$ if and only if the point $(\langle v_1 \rangle, \ldots,
\langle v_p \rangle)$ is in the zero locus $Z_f$ of the section $s_f$. In
\cite{FriOtt}, this is used to compute the cardinality of $Z_f$ for $f$
sufficiently general as the top Chern class of $\cE$. Our current task is
different: we want to show that, if we assume the triangle inequalities of
Theorem~\ref{thm:Main} and that $f$ is sufficiently general, the linear
span $\langle Z_f \rangle$ equals the projectivised critical space $\Pe
H_f$; this is the second part of Theorem~\ref{thm:Main}.

\subsection*{Bott's formulas and a consequence}

Our central tool will be the following formulas for the cohomology
of vector bundles over projective spaces \cite{OSSBott}. Recall that
$\Omega_{\Pe^n}^r (k)$ is the $\cO(k)$-twisted sheaf of differential
$r$-forms on $\Pe^n$.

\begin{lem}[Bott's formulas]
For $q,n,r \in \ZZ_{\geq 0}$ and $k \in \ZZ$ we have
\[ h^q(\Pe^n,\Omega_{\Pe^n}^r(k)) 
=
\begin{cases}
\binom{k+n-r}{k} \binom{k-1}{r} & \text{if $q=0 \leq r \leq n$ and
$k>r$,} \\
1 & \text{if $0 \leq q=r \leq n$ and $k=0$,}\\
\binom{-k+r}{-k} \binom{-k-1}{n-r} & \text{if $q=n \geq r \geq 0$ and
$k<r-n$, and}\\
0 & \text{otherwise.}
\end{cases}
\]
\end{lem}

A consequence featuring the triangle inequalities of
Theorem~\ref{thm:Main} is the following. 

\begin{lem} \label{lem:ConsBott}
Suppose that $n_\ell \leq \sum_{i \neq \ell} n_{i}$ holds for all $\ell$ with
$d_\ell=1$. Let $k \geq 2$ be an integer, $q_1,\ldots,q_p$ be nonnegative
integers with $\sum_{\ell=1}^p q_\ell < k$ and $r_1,\ldots,r_p$ be
nonnegative integers with $\sum_{\ell=1}^p r_\ell =k$. Then
\[ \bigotimes_{\ell=1}^p H^{q_\ell}(\Pe V_\ell,
\Omega_{\Pe V_\ell}^{r_\ell}(-d_\ell(k-1)+2r_\ell)) =0. \]
\end{lem}

\begin{proof}
Assume that all factors in the tensor product are nonzero.

First, if all of the factors were nonzero by virtue of the second and
third line in Bott's formulas, then we would have $q_\ell \geq r_\ell$ for
all
$\ell$, and hence $k>\sum_\ell q_\ell \geq \sum_\ell r_\ell=k$, a contradiction.

So some factor is nonzero by virtue of the first line in Bott's
formulas;
without loss of generality this is the first factor. Hence we have
$q_1=0
\leq r_1 \leq n_1$ and $-d_1(k-1)+2r_1>r_1$. This last inequality
reads
$r_1>d_1(k-1)$. Combining this with $\sum_\ell r_\ell=k$ and the fact that
$d_1$
is a positive integer, we find that
$r_1=k$, $d_1=1$, and $r_\ell=0$ for $\ell>1$.
In particular, there are no $\ell>1$ for which the first line in Bott's
formulas applies.

For any $\ell > 1$, if the second line applies, then
$0=r_\ell=q_\ell=-d_\ell(k-1)+2r_\ell$, which contradicts that both
$d_\ell$ and
$k-1$
are positive. Hence the third line applies for all $\ell>1$, and in
particular we have $q_\ell=n_\ell$. But then
\[ n_1 \geq r_1 = k > \sum_{l=1}^p q_l =
\sum_{l=2}^p
n_l,\]
which together with $d_1=1$ contradicts the triangle inequality in the
lemma.
\end{proof}

\subsection*{Vanishing cohomology}

The vanishing result in this subsection uses Lemma~\ref{lem:ConsBott}
and the following version of K\"unneth's formula.

\begin{lem}[K\"unneth's formula]
For vector bundles $\mathcal{G}_\ell$ on $\Pe V_\ell$ for
$\ell=1,\ldots,p$ and a nonnegative integer $q$ we have 
\[ H^q(\Pe X,\bigotimes_\ell \pi_\ell^* \mathcal{G}_\ell)
\cong \bigoplus_{q_1+\ldots+q_p=q} \bigotimes_\ell H^{q_\ell}(\Pe
V_\ell,\mathcal{G}_\ell), \]
where the sum is over all $p$-tuples of nonnegative integers summing
to $q$.
\end{lem}

\begin{lem}\label{lema3-SV}
Suppose that $n_\ell \leq \sum_{i\neq \ell} n_i$ holds for all $\ell$ such
that $d_\ell=1$. Let $k \geq 2$ be an integer and $q \in \{0,\ldots,k-1\}$. Then
we have
\[ 
H^q(\Pe X,(\bigwedge^{k}\cE^*) \otimes \cO(d_1,\ldots,d_p))=0. 
\]
\end{lem}

\begin{proof}
First, 
\[ \cE^*=\bigoplus_{\ell=1}^{p}(\pi_\ell^*\mathcal{Q_\ell}^*)\otimes
\cO(-d_1,\ldots,-d_{\ell-1},-(d_\ell-1),-d_{\ell+1},\ldots,-d_p).
\]
A well-known formula for $k$-th wedge power of a direct sum yields
\[ 
\bigwedge^k\cE^*=\bigoplus_{r_1+\ldots+r_p=k}
\bigotimes_\ell 
\bigwedge^{r_\ell}(\pi_\ell^*\mathcal{Q_\ell}^*\otimes \cO(-d_1,
\ldots,-(d_\ell-1),\ldots,-d_p)). \]
Using 
$\bigwedge^r(\mathcal{F}\otimes
\cO(\omega))=(\bigwedge^r\mathcal{F})(\red{r}\omega)$,
$\cQ^*=\Omega^1(1)$, and
$\bigwedge^r (\Omega^1(1))=\Omega^{r}(r)$, 
we obtain
\[ 
\bigwedge^k\cE^*=\bigoplus_{r_1+\ldots+r_p=k}
\bigotimes_{\ell} (\pi_\ell^*\Omega_{\Pe V_\ell}^{r_\ell}(r_\ell)\otimes 
\cO(-r_\ell d_1,\ldots,-r_\ell(d_\ell-1),\ldots,-r_\ell d_p).
\]
Twisting by $\cO(d_1,\ldots,d_p)$, regrouping in
each projection, and using $\sum_\ell r_\ell=k$ we find:
\[ 
(\bigwedge^k\cE^*) \otimes \cO(d_1,\ldots,d_p)=\bigoplus_{r_1+\ldots+r_p=k}
\bigotimes_{\ell} \pi_\ell^*\Omega_{\Pe V_\ell}^{r_\ell}(-d_\ell(k-1)+2 r_\ell). 
\]
To compute $H^q$ of each summand we apply K\"unneth's formula, and
obtain subsummands which are exactly of the form in
Lemma~\ref{lem:ConsBott}, hence zero.
\end{proof}

\subsection{Comparing $\Pe H_f$ and $\langle Z_f \rangle$}

Assume that $f$ is sufficiently general in $T$.  By the first subsection
of this section and by Proposition~\ref{prop:RankOne}, $Z_f$ is contained
in the projectivised critical space $\Pe H_f$, hence so is $\langle
Z_f \rangle$.  Our goal now is to show that $\langle Z_f \rangle$ is
{\em equal} to $\Pe H_f$ and to compute its dimension. Both of these
goals are achieved through the following lemma. The section $s_f$ of
$\cE$ yields a homomorphism $\cE^* \to \cO$ of sheaves whose image is
contained in the ideal sheaf $\cI_{Z_f}$ of the zero locus of $s_f$.

\begin{lem} \label{lem:linforms}
Assume that for each $\ell \in [p]$ we have $n_\ell \leq \sum_{i \neq
\ell} n_i$ and that $f$ is sufficiently general. Then the induced
homomorphism $\cE^* \otimes \cO(d_1,\ldots,d_p) \to \cI_{Z_f} \otimes
\cO(d_1,\ldots,d_p)$ induces an isomorphism at the level of global
sections.
\end{lem}

The following proof can be shortened considerably using spectral
sequences, but we found it more informative in its current form. To make
the formulas more transparent, we write $H^q(.)$ instead of $H^q(\Pe
X,.)$ everywhere.

\begin{proof}
To establish the desired isomorphism
\[ H^0(\cE^* \otimes \cO(d_1,\ldots,d_p)) 
\cong H^0(\cI_{Z_f} \otimes \cO(d_1,\ldots,d_p)) \]
we use the following Koszul complex (see, e.g., \cite[Chapter III,
Proposition 7.10A]{Hartshorne}):
\[ 0=\bigwedge^{N+1} \cE^* \rightarrow\bigwedge^N\cE^*\rightarrow\cdots\rightarrow\bigwedge^2
\cE^*\rightarrow \cE^*\rightarrow \cI_Z\rightarrow 0. \]
Letting $\cF_k$ be the quotient of $\bigwedge^k \cE^*$ by the image of
$\bigwedge^{k+1} \cE^*$, this yields the short exact sequence
\[ 0 \to \cF_2 \to \cE^* \to \cI_Z \to 0. \]
Tensoring with $\cO(d_1,\ldots,d_p)$ yields the short exact sequence
\[ 0 \to \cF_2 \otimes \cO(d_1,\ldots,d_p) \to 
\cE^* \otimes \cO(d_1,\ldots,d_p) \to 
\cI_Z \otimes \cO(d_1,\ldots,d_p) \to 0, \]
and this gives a long exact sequence in cohomology beginning with 
\begin{align*} 
0 &\to H^0(\cF_2 \otimes \cO(d_1,\ldots,d_p)) \to 
H^0 (\cE^* \otimes \cO(d_1,\ldots,d_p)) \to 
H^0 (\cI_Z \otimes \cO(d_1,\ldots,d_p))\\ & \to 
H^1 (\cF_2 \otimes \cO(d_1,\ldots,d_p)) \to 
\end{align*}
So to obtain the desired isomorphism we want that 
\[ H^q(\cF_2 \otimes \cO(d_1,\ldots,d_p))=0 \quad \text{for $q=0,1$.}\]
For each $k=2,\ldots,N$ we have the short exact sequence
\[ 0 \to \cF_{k+1} \to \bigwedge^k \cE^* \to \cF_k \to 0 \]
which yields the long exact sequence
\begin{align*} 
& \to 
H^{k-2} (\bigwedge^k \cE^* \otimes \cO(d_1,\ldots,d_p)) \to 
H^{k-2} (\cF_k \otimes \cO(d_1,\ldots,d_p)) \to 
H^{k-1} (\cF_{k+1} \otimes \cO(d_1,\ldots,d_p)) \\ & \to 
H^{k-1} (\bigwedge^k \cE^* \otimes \cO(d_1,\ldots,d_p)) \to 
H^{k-1} (\cF_k \otimes \cO(d_1,\ldots,d_p)) \to 
H^k (\cF_{k+1} \otimes \cO(d_1,\ldots,d_p)) \to 
\end{align*}
Using Lemma~\ref{lema3-SV} the two leftmost spaces are zero, so that
\begin{align*} 
H^{k-2} (\cF_k \otimes \cO(d_1,\ldots,d_p)) 
&\cong H^{k-1} (\cF_{k+1} \otimes \cO(d_1,\ldots,d_p)) \text{ and}\\
H^{k-1} (\cF_k \otimes \cO(d_1,\ldots,d_p)) 
&\subseteq H^{k} (\cF_{k+1} \otimes \cO(d_1,\ldots,d_p)) \text{ and}\\
\end{align*}
Hence using that $\cF_{N+1}=0$ we find that 
\begin{align*} 
H^0 (\cF_2 \otimes \cO(d_1,\ldots,d_p)) &\cong \cdots \cong H^{N-1}(\cF_{N+1} \otimes
\cO(d_1,\ldots,d_p)) = 0 \text{ and}\\
H^1 (\cF_2 \otimes \cO(d_1,\ldots,d_p)) &\subseteq \cdots \subseteq H^{N}(\cF_{N+1} \otimes
\cO(d_1,\ldots,d_p)) = 0, 
\end{align*}
as desired.
\end{proof}

\begin{prop} \label{prop:dim}
Suppose that for each $\ell \in [p]$ we have $n_\ell \leq \sum_{i \neq
\ell} n_i$ and that $f$ is sufficiently general. Then $\langle Z_f
\rangle=\Pe H_f$ and $\codim_T H_f=\sum_\ell \binom{n_\ell+1}{2}$.
\end{prop}

\begin{proof}
Since $\Pe X$ is embedded by $\cO(d_1,\ldots,d_p)$, the
space of linear forms on $T$ vanishing on $Z_f$ is $H^0(\cI_{Z_f} \otimes
\cO(d_1,\ldots,d_p))$. By Lemma~\ref{lem:linforms} this space is
isomorphic to
\begin{align*} 
&H^0(\cE^* \otimes \cO(d_1,\ldots,d_p)) =
\bigoplus_\ell H^0(\pi_\ell^* \cQ_l^* \otimes
\cO(0,\ldots,1,\ldots,0))\\
&=\bigoplus_\ell H^0(\pi_{\ell^*}(\Omega^1_{\Pe
V_\ell}(2)))
=\bigoplus_\ell H^0(\Pe V_\ell, \Omega^1_{\Pe V_\ell}(2)),
\end{align*}
which by the first line in Bott's formulas has dimension $\sum_{\ell}
\binom{n_\ell+1}{2}$. This means that $\codim_{\Pe T} \langle Z_f
\rangle=\sum_\ell \binom{n_\ell+1}{2}$, so the second statement in the
proposition follows from the first statement. 

To establish the first statement, we spell out the map
\[ H^0(\Pe V_\ell, \cQ_\ell^* \otimes \cO(1))=H^0(\Pe V_\ell, \Omega_{\Pe V_\ell}^1 (2)) \to H^0(\cI_{Z_f} \otimes
\cO(d_1,\ldots,d_p)) \]
in greater detail. The space on the left is canonically $(\bigwedge^2
V_\ell)^*$, and an element $\xi$ in this space is mapped to the linear form
$T \to \CC, g \mapsto \xi([f|g]_\ell)$. As $\ell$ varies, these
are precisely the linear forms that cut out $H_f$. This
proves that $\Pe H_f = \langle Z_f \rangle$. 
\end{proof}

\begin{remark}
In general, for the equality $\langle Z_f \rangle=\Pe H_f$ we only need
that the linear equations cutting out $\Pe H_f$ also cut out $Z_f$,
i.e., we only need that the linear map in Lemma~\ref{lem:linforms}
is surjective.  One might wonder whether this surjectivity remains true
when the triangle inequalities fail. In the case of $(n_1+1) \times
(n_2+1)$-matrices, it does indeed---there we already knew the critical
rank-one approximations span the critical space---but for $p=3$ and
$2 \times 2 \times 4$-tensors (so that $n_3=3>1+1=n_1+n_2$) 
the space $\langle Z_f \rangle$ has dimension $6$ while computer
experiments suggest that the space $\Pe H_f$ has dimension $7$
, hence the surjectivity fails.  Still, in these experiments,
$f$ itself seems to lie in the span of $Z_f$. This leads to the open
problem whether our analogue of the Spectral Theorem and the Eckart-Young
Theorem persists when the triangle inequalities fail.
\end{remark}

\section{Proofs of the main results}
\label{sec:Final}

\begin{proof}[Proof of Theorem~\ref{thm:Main}]
The first statement is Proposition~\ref{prop:HigherRank}; the second
and \red{third} statement are Proposition~\ref{prop:dim}. \red{The last statement follows from Remark \ref{rem:fincritical}.}
\end{proof}

\begin{proof}[Proof of Corollaries~\ref{cor:tensor} and~\ref{cor:sym}.]
If $g$ is a real tensor of real rank at most $k$ closest to $f$, then one
can write it as $x_1+\ldots+x_k$ with $x_1,\ldots,x_k$ real points of $X$.
In particular, all of these points are non-isotropic, and the argument
of Proposition~\ref{prop:HigherRank} applies. Hence $g$ lies in $H_f$.
Now the result follows from Proposition~\ref{prop:dim}. The argument
applies, in particular, to $k$ equal to the rank of $f$, which gives
the last statement of the corollaries.
\end{proof}

Note that, if $f$ is {\em any} real tensor, then any real tensor of real
rank at most $k$ closest to $f$ lies in $H_f$ by the argument above. Only
for the conclusion that it lies in the span of the complex critical
rank-one tensors of $f$ do we use that $f$ is sufficiently general. We
do not know whether this generality is really needed.

Also note that we do not shed new light on the question of when for
sufficiently general $f$ there {\em exists} a closest real tensor of
rank at most $k$. For an update on the complex case see \cite{Qi}.

\end{document}